\documentclass[12pt]{amsart}
\usepackage{amssymb,latexsym}
\usepackage{graphicx}
\usepackage{url}		

\newcommand\scalemath[2]{\scalebox{#1}{\mbox{\ensuremath{\displaystyle #2}}}}

\newcommand{\Q}{{\mathbb Q}}
\newcommand{\C}{{\mathbb C}}

\newcommand{\Z}{{\mathbb Z}}
\DeclareRobustCommand{\stirling}{\genfrac\{\}{0pt}{}}

\theoremstyle{plain}
\numberwithin{equation}{section}
\newtheorem{thm}{Theorem}[section]
\newtheorem{theorem}[thm]{Theorem}
\newtheorem{lemma}[thm]{Lemma}

\newtheorem{definition}[thm]{Definition}
\newtheorem{corollary}[thm]{Corollary}
\newtheorem*{rem}{Remark}

\begin{document}
\setcounter{page}{1}

\title{Multivariate Fibonacci-Like Polynomials and Their Applications}

\author{Sejin Park}
\address{Department of Mathematics\\
                Brown University\\
                Providence, RI, 02912}
\email{sejin\_park@brown.edu}
\author{Etienne Phillips}
\address{Department of Mathematics\\
                North Carolina State University\\
                Raleigh, NC, 27695}
\email{ecphill6@ncsu.edu}
\author{Peikai Qi}
\address{Department of Mathematics\\
               Michigan State University\\
               East Lansing, MI, 48824}
\email{qipeikai@msu.edu}
\author{Zhan Zhan}
\address{Department of Mathematics\\
               Michigan State University\\
               East Lansing, MI, 48824}
\email{zhanzha2@msu.edu}
\author{Ilir Ziba}
\address{Department of Mathematics\\
                University of Michigan\\
                Ann Arbor, MI, 48109}
\email{zibai@umich.edu}

\thanks{The authors would like to thank Aklilu Zeleke for his suggestions and mentorship. This work is partially supported by grants from the National Science Foundation (NSF Award No. 1852066) and Michigan State University's SURIEM REU}

\begin{abstract}
The Fibonacci polynomials are defined recursively as $f_{n}(x)=xf_{n-1}(x)+f_{n-2}(x)$, where $f_0(x) = 0$ and $f_1(x)= 1$. We generalize these polynomials to an arbitrary number of variables with the $r$-Fibonacci polynomial.
We extend several well-known results such as the explicit Binet formula and a Cassini-like identity, and use these to prove that the $r$-Fibonacci polynomials are irreducible over $\mathbb{C}$ for $n \geq r \geq 3$. Additionally, we derive an explicit sum formula and a generalized generating function. Using these results, we establish connections to ordinary Bell polynomials, exponential Bell polynomials, Fubini numbers, and integer and set partitions. 
\end{abstract}

\maketitle

\section{Introduction} \label{form}
The standard Fibonacci polynomials, which have been the subject of extensive study, are defined by the recursion $$f_n(x)=xf_{n-1}+f_{n-2}$$
where $f_0(x)=0$ and $f_1(x)=1$. This family of polynomials has also been generalized to a variation in two variables: 
$$F_n(x,y) = xF_{n-1}(x,y) + yF_{n-2}(x,y)$$

In \cite{ArsComb}, some explicit forms are presented are used to derive some combinatorial applications. Moreover, the irreducibility of these polynomials is studied in \cite{Divisibility}. In this paper, we generalize some of the results for the two-variate Fibonacci polynomials to $r$ variables and present some combinatorial applications.

\begin{definition}
The \textbf{\textit{r}-Fibonacci Polynomial} is defined as
$$F_n^{[r]}(x_1,x_2,...x_r)=\begin{cases} 0 &0\leq n<r-1 \\ 1 & n=r-1 \\ \sum_{i=1}^rx_iF^{[r]}_{n-i} & n \geq r \end{cases} $$
\end{definition}

Notice that specific cases of the $r$-Fibonacci polynomial generate well-known sequences of numbers and polynomials. For example, $F_n^{[2]}(x,1)$ are the aforementioned standard Fibonacci polynomials, and \newline $F_n^{[r]}(1,1,\dots, 1)$ gives the famous Fibonacci, Tribonacci and Tetranacci number sequences when $r= 2, 3$ and $4$ respectively.

We also generalize this family of polynomials further to include a family with unfixed initial conditions.

\begin{definition}
The \textbf{Generic \textit{r}-Fibonacci Polynomial} is defined as

$$ \mathcal{F}^{[r]}_n=
\begin{cases}
\ell_n(x_1,...x_r) & 0\leq n\leq r-1 \\ \sum_{i=1}^rx_i\mathcal{F}^{[r]}_{n-i} & n\geq r \end{cases} $$

Where $\ell_i\in\Q[x]$ is an arbitrary polynomial.

\end{definition}

In Section \ref{sec:binet}, we generalize Binet's formula and Cassini's identities to the $r$-Fibonacci polynomials and use the former to give a complete classification of the irreducibility of all $r$-Fibonacci polynomials. In Section \ref{sec:part}, we draw a connection between the $r$-Fibonacci polynomials and integer partitions and use this to derive an explicit combinatorial formula. Then, we study the generating functions of $r$-Fibonacci polynomials in Section \ref{sec:gen} and derive several identities relating to combinatorial sequences, such as the Fubini numbers, Pell numbers, and Stirling numbers of the second kind. Finally, using the combinatorial sum formula and generating functions, we prove a relationship to Bell polynomials in Section \ref{sec:bell}.

\section{Explicit Forms and Identities} \label{sec:binet}

The standard Fibonacci polynomials have a well-known closed form known as a Binet form (e.g., see \cite{Roots}). This formula is given by

\begin{equation}
    f_n(x)=F_n^{[2]}(x,1)=\frac{\alpha^{n}-\beta^n}{\alpha-\beta}
\end{equation}
with 
$$\alpha=\frac{x+\sqrt{x^2+4}}{2} \ \ \ \ \ \ \beta=\frac{x-\sqrt{x^2+4}}{2} $$

In \cite{ArsComb}, this formula is generalized to the two-variate case. In this section we generalize the Binet formula to the r-Fibonacci polynomials.

\subsection{The Generalized Binet Form} \

In \cite{SpectralTheory}, a method of deriving the Binet Form of the Fibonacci numbers via spectral-theoretic methods is presented. This method can be generalized to the generic $r$-Fibonacci polynomials.

Observe that for $n\geq r-1$, the following matrix product identity holds:
\begin{equation}\begin{pmatrix}
    x_1 & x_2 & \dots & x_r \\
    1 & & 0 & 0 \\
    & \ddots & & \vdots \\
    0 & & 1 & 0
\end{pmatrix} 
\begin{pmatrix} 
\mathcal{F}_n^{[r]}\\
\mathcal{F}_{n-1}^{[r]}\\
\vdots \\
\mathcal{F}_{n-r+1}^{[r]}\\
\end{pmatrix} 
=
\begin{pmatrix} 
\mathcal{F}_{n+1}^{[r]}\\
\mathcal{F}_{n}^{[r]}\\
\vdots \\
\mathcal{F}_{n-r+2}^{[r]}\\
\end{pmatrix} 
\end{equation}

Equation (2.2) gives the following lemma.

\begin{lemma}
Let $M=\scalemath{0.7}{\begin{pmatrix}
    x_1 & x_2 & \dots & x_r \\
    1 & & 0 & 0 \\
    & \ddots & & \vdots \\
    0 & & 1 & 0
\end{pmatrix}}$.
Then \begin{equation}M^{n-r+1} 
\begin{pmatrix} 
\mathcal{\ell}_{r-1}^{[r]}\\
\mathcal{\ell}_{r-2}^{[r]}\\
\vdots \\
\mathcal{\ell}_{0}^{[r]}\\
\end{pmatrix} 
=
\begin{pmatrix} 
\mathcal{F}_{n}^{[r]}\\
\mathcal{F}_{n-1}^{[r]}\\
\vdots \\
\mathcal{F}_{n-r+1}^{[r]}\\
\end{pmatrix}  \end{equation}
\end{lemma}
\begin{proof}
Follows by induction on $n$.
\end{proof} 

The next lemma shows how  $M$ can be diagonalized.

\begin{lemma} Suppose $x_1,... x_r\in \mathbb{C}$ such that $M$ has $r$ distinct eigenvalues, and let $\{\lambda_k\}_{1\leq k \leq r}$ be the eigenvalues of $M$ and $D$ be the diagonal matrix with entries $\lambda_k$. Then $M=SDS^{-1}$, where $S_{i,j}=\lambda_j^{i-1}$ and $S^{-1}_{i,j}=\sigma_{i,j}$ is given by:

$$\sigma_{i,j}=\begin{cases} (-1)^{r-j}\bigg(\frac{\operatorname*{\sum} \limits_{{\substack{1\leq m_1 < m_2<\cdots<m_{j-1}\leq r \\ m_1,\cdots,m_{j-1}\neq i}}} \lambda_{m_1}\lambda_{m_2}\cdots \lambda_{m_{j-1}}}
{\operatorname*{\prod} \limits_{{\substack{1\leq m \leq r \\ m\neq i}}}(\lambda_{m}-\lambda_{i})} \bigg) & 1<j\leq r \\ 
\bigg( \operatorname*{\prod} \limits_{{\substack{1\leq m \leq r \\ m\neq i}}}\lambda_{i}-\lambda_{m} \bigg)^{-1}& j=1 \end{cases} $$
\end{lemma}

\begin{proof}
Because $\lambda_k$ are solutions to the characteristic equation $\lambda_k^r-x_1\lambda_k^{r-1}-...-x_r=0$,

$${\begin{pmatrix}
    x_1 & x_2 & \dots & x_r \\
    1 & & 0 & 0 \\
    & \ddots & & \vdots \\
    0 & & 1 & 0
\end{pmatrix}} \begin{pmatrix}\lambda_k^{r-1}\\\lambda_k^{r-2}\\\vdots\\1 \end{pmatrix}= \begin{pmatrix}\sum_{i=1}^rx_i\lambda_k^{r-i}\\\lambda_k^{r-1}\\\vdots\\\lambda_k \end{pmatrix} = \begin{pmatrix}\lambda_k^r\\\lambda_k^{r-1}\\\vdots\\\lambda_k \end{pmatrix}.$$
The columns of the matrix $S$ are eigenvectors of $M$. Because $S$ is the Vandermonde matrix, the entries of its inverse are already well known to be given by $\sigma_{i,j}$ (e.g., exercise 40 in \cite{knuth_1997}).
\end{proof} 

Using the above lemma we derive an explicit form for the generic $r$-Fibonacci polynomials.

\begin{theorem} The explicit form of the Generic $r$-Fibonacci Polynomials is given by
$$\mathcal{F}_n^{[r]}=\sum_{i=1}^{r}\bigg(\lambda_i^{n}\sum_{j=1}^r(\sigma_{i,j}\ell_{r-j})\bigg) $$
\end{theorem}

\begin{proof}
By \textbf{Lemmas 2.1} and \textbf{2.2}, we obtain the following vector-matrix product from equation (2.3):

\tiny{$$ \hspace{-10mm} \begin{pmatrix}
    \lambda_1^{r-1} & \lambda_2^{r-1} & \cdots & \lambda_r^{r-1} \\
    \lambda_1^{r-2} & \lambda_2^{r-2} & \cdots & \lambda_r^{r-2}\\
    \vdots & \vdots & \cdots & \vdots \\
    \lambda_1 & \lambda_2 & \cdots & \lambda_r \\
    1 & 1 & \cdots & 1
\end{pmatrix}
\begin{pmatrix}
    \lambda_1^{n-r+1} & & &0 \\
    & \lambda_2^{n-r+1} & & &\\
    & & \ddots & & \\
    0& & & \lambda_r^{n-r+1}
\end{pmatrix}
\begin{pmatrix}
    \sigma_{1,1} & \cdots & \sigma_{1,r}\\
    \vdots &  & \vdots \\
    \sigma_{r,1} & \cdots & \sigma_{r,r}
\end{pmatrix}
\begin{pmatrix} 
    \ell_{r-1} \\
    \vdots \\
    \ell_{0} \\
\end{pmatrix}$$
$$
=\begin{pmatrix} 
\mathcal{F}_{n}^{[r]}\\
\mathcal{F}_{n-1}^{[r]}\\
\vdots \\
\mathcal{F}_{n-r+1}^{[r]}\\
\end{pmatrix}
$$}

\normalsize{Multiplying the matrices together gives}

\small{$$\begin{pmatrix}
    \lambda_1^n & \lambda_2^n & \cdots & \lambda_r^n \\
    \vdots & \vdots & \cdots & \vdots
\end{pmatrix}
\begin{pmatrix}
    \sum_{j=1}^r\sigma_{1,j}\ell_{r-j}\\
    \sum_{j=1}^r\sigma_{2,j}\ell_{r-j}\\
    \vdots\\
    \sum_{j=1}^r\sigma_{r,j}\ell_{r-j}
\end{pmatrix}
=
\begin{pmatrix}
    \displaystyle \sum_{i=1}^r\bigg(\lambda_i^{n}\sum_{j=1}^r\sigma_{i,j}\ell_{r-j}\bigg)\\
    \vdots
\end{pmatrix}
=\begin{pmatrix} 
\mathcal{F}_{n}^{[r]}\\
\vdots
\end{pmatrix}
$$}

Comparing the entries in the above matrices gives the desired equality.

\end{proof}

Note that known Binet-like formula emerges as special cases of \textbf{Theorem 2.3}, so it make sense to refer to the previous theorem as a generalized Binet form. We obtain the following result about the $r$-Fibonacci polynomials

\begin{corollary} The Binet form of the $r$-Fibonacci polynomial is given by 
$$ F_n^{[r]}=\sum_{i=1}^r \frac{\lambda_i^n}{\displaystyle \prod_{1\leq m \leq r \atop m\neq i}(\lambda_i-\lambda_m)}=\sum_{m_1+m_2+...+m_r=n-r+1}\lambda_1^{m_1}\lambda_2^{m_2}...\lambda_r^{m_r}$$
\end{corollary}
\begin{proof}
The first equality is given by applying \textbf{Theorem 2.4}, and the second can be verified by straightforward algebraic manipulation.
\end{proof}

For $r=2$, we obtain the equation (2.1).

\begin{rem}
The generalized Binet form cannot be solved explicitly in terms of radicals for $r\geq 5$ because the characteristic equation of $M$ is an arbitrary monic polynomial of degree $r$.
\end{rem}
\subsection{Generalization of Cassini's Identity}\hspace{1mm}

In the case for $r=2$, $x_2=1$, then Cassini's identity states 

$$f_{n-1}f_{n+1}-(f_{n})^2=(-1)^n$$

Using a similar strategy as in the above results, we find a generalization of Cassini's identity. This can be derived by taking the determinants on both sides of the matrix identity for the standard Fibonacci polynomials, which is given by
$$\begin{pmatrix} x & 1 \\ 1 & 0 \end{pmatrix}^{n-2}=\begin{pmatrix}
    f_{n+1} & f_{n} \\ f_{n} & f_{n-1}
\end{pmatrix}$$
We generalize this identity as follows.

\begin{theorem} For $n\geq 2r-2$

$$\det \begin{pmatrix} 
    F_{n-r+1}^{[r]} & \dots & F_{n}^{[r]} \\
    \vdots & \ddots & \vdots \\
    F_{n-2r+2}^{[r]} & \dots & F_{n-r+1}^{[r]}
\end{pmatrix}=(-1)^{n(r+1)}x_r^{n-2r+2} $$
\end{theorem}
\begin{proof}
Observe that by extending equation (2.3), we get the following matrix product identity:

$$M^{n-2r+2 }\begin{pmatrix}
      F_{r-1}^{[r]} & \dots & F_{2r-2}^{[r]} \\
    \vdots & \ddots & \vdots \\
    F_{0}^{[r]} & \dots & F_{r-1}^{[r]}
\end{pmatrix}=  \begin{pmatrix}  F_{n-r+1}^{[r]} & \dots & F_{n}^{[r]} \\
    \vdots & \ddots & \vdots \\
    F_{n-2r+2}^{[r]} & \dots & F_{n-r+1}^{[r]} \end{pmatrix}$$
The desired equality follows from taking the determinant of both sides.
\end{proof}

For example, in the case $r=3$ we have the following result for $n\geq4$

$$(F_{n-2}^{[3]})^2-F_{n-1}^{[3]}F^{[3]}_{n-2}F^{[3]}_{n-3}+F_n^{[3]}(F^{[3]}_{n-3})^2-F^{[3]}_nF^{[3]}_{n-2}F^{[3]}_{n-4}+F^{[3]}_{n-1}F^{[3]}_{n-4}=x_3^{n-4} $$

\subsection{Irreducibility}\hspace{1mm}

In \cite{Divisibility}, it is shown that, for the $r=2$ case, the $r$-Fibonacci polynomials are irreducible if and only if $n$ is prime. Utilizing the generalized Binet-like formula, we prove a stronger result for the $r$-Fibonacci polynomials when $r \geq 3$. First we prove the following lemma.

\begin{theorem}
$F_n^{[r]}(x_1,...x_r)$ is irreducible over $\C$ for all $n\geq r \geq 3$ 
\end{theorem}

\begin{proof}

First we show that $F_{n+r-3}^{[r]}$ is irreducible over $\C$ whenever $F_n^{[3]}$ is irreducible over $\C$.

Through induction on $n$, we can see the polynomial $F_n^{[r]}(x_1,x_2^2,x_3^3,...x_r^r)$ is a homogeneous polynomial of degree $n-r+1$. 

Suppose $F_{n+r-3}^{[r]}$ is reducible over $\C$. Then $F_{n+r-3}^{[r]}(x_1,x_2^2,...x_r^r)=g(x_1,...x_r)h(x_1,...x_r)$ where $g,h\in\C[x_1,...x_r]$ are non-constant, homogeneous polynomials. Thus \newline $g(x_1,x_2,x_3,0,...0)h(x_1,x_2,x_3,0,...0)=F_n^{[3]}(x_1,x_2^2,x_3^3)$ where \newline
$g(x_1,x_2,x_3,0,...0)$ and $h(x_1,x_2,x_3,0...0)$ cannot be constants, so \newline $F_n^{[3]}(x_1,x_2^2,x_3^3)$ is therefore reducible. This implies that $F_n^{[3]}$ is reducible over $\C$. Thus, it is enough to show that $F_{n}^{[3]}$ is irreducible for $n\geq 3$.

By \textbf{Corollary 2.5}, the Binet form of the $3$-Fibonacci polynomials is given by 
\begin{equation}
F_n^{[3]}(x_1,x_2,x_3)=\sum_{i+j+k=n-2}\lambda_1^i\lambda_2^j\lambda_3^k\end{equation}
where $\lambda_i$ are the solutions to the characteristic polynomial $z^3-x_1z^2-x_2z-x_3=0$. If $\lambda_1=s$, $\lambda_2=t$, $\lambda_3=u$, we can solve for $x_i$ in terms of $s, t, u$. So long as there are 3 distinct eigenvalues and equation (2.4) holds.
\footnotesize $$z^3-x_1z^2-x_2z-x_3=(z-s)(z-t)(z-u)=z^3-(s+t+u)z^2-(-st-su-tu)z - stu$$
\normalsize $$\implies x_1=s+t+u \hspace{5mm} x_2=-st-su-tu \hspace{5mm} x_3=stu $$

Since this parameterization of $x_1, x_2$, and $x_3$  has no algebraic relationship, the resulting polynomial is irreducible if and only if $F_n^{[3]}$ is irreducible.

The irreducibility of $\displaystyle \sum_{i+j+k=n-2}s^it^ju^k$ over $\C$ for all $n\geq3$ is proved in \cite{brunault_2912}, and thus $F_n^{[3]}$ is irreducible for all $n\geq3$.
\end{proof}

\section{Partitions}\label{sec:part}
In this section, we connect the $r$-Fibonacci polynomials to integer partitions and derive an associated explicit combinatorial formula for $F_n^{[r]}$.

To begin, an integer partition is denoted as $(1^{a_1}, 2^{a_2}, 3^{a_3},...)$, where $a_i$ is the multiplicity of $i$ in the multiset. Let $P_n(r)$ be the set of all partitions of $n$ with elements no greater than $r$. Lastly, we refer to partitions of size $n$ as $\rho \in P_n(r)$; $|\rho| = n$.

To demonstrate the connection between $r$-Fibonacci polynomials and $P_n(r)$, we define the following function.

\begin{definition}
Let $G = \displaystyle \sum_{(a_1,a_2,\dots,a_r)\in I} c_Ix_1^{a_1}x_2^{a_2}...x_r^{a_r} \in\Z[x_1,...x_r]$. Then $\Omega$ is defined by
$$\Omega(G) = \{(1^{a_{1}},2^{a_{2}},...,r^{a_{r}}) \; | \; (a_1,a_2,\dots,a_r)\in I\}$$
\end{definition}

For example, $\Omega(x_1^3+x_1x_2)=\Omega(x_1^3)\cup\Omega(x_1x_2)=\{(1^3),(1,2)\}$ generates two partitions of 3.

Note that if $G,H\in\Z[x_1,x_2,\dots,x_r]$ have positive integer coefficients, $$\Omega(G + H) = \Omega(G) \cup \Omega(H).$$

\begin{theorem}
For $n\geq r$, then
$\Omega(F_n^{[r]})=P_{n-r+1}(r)$ for $n\geq r$.
\end{theorem}
\begin{proof}
We proceed with strong induction on $n$ with the base case given by $n=r$: $$\Omega(F_{r}^{[r]})=\Omega(x_1)=\{(1)\}=P_1(r) $$

Assume equality holds for integers between $r$ and $n$. By the properties of $\Omega$, 

\begin{equation}
    \Omega(F_{n+1}^{[r]})=\bigcup_{i=1}^{r}\Omega(x_iF_{n-i+1}^{[r]})
\end{equation}

Observe that by the inductive hypothesis, for some partition $\rho\in\Omega(x_iF_{n-i+1}^{[r]})$, $|\rho|=i+(n-r+2-i)=n-r+2$. Therefore all partitions in $\Omega(F_{n+1}^{[r]})$ have size $k-r+2$, so $\Omega(F_{k+1}^{[r]})\subseteq P_{n-r+2}(r)$. 

We must now show that $P_{n-r+2}(r)\subseteq\Omega(F_{n+1}^{[r]})$. Let $\rho_{i}\in P_{k-r+2}(r)$. Then $\rho_{i}=(i,\rho')$ where $1\leq i \leq r$ and $\rho'\in P_{k-r+2-i}(r)=\Omega(F^{[r]}_{k-r+2-i})$. Therefore, $\rho_i\in\Omega(x_iF^{[r]}_{n-k+1-i})$. Thus, by equation (2.4), $\rho_i\in P_{k-r+2}(r)$.

By double inclusion, $P_{k-r+2}(r)=\Omega(F^{[r]}_{k+1})$.

\end{proof}

Using this characterization, we derive the following combinatorial formula for the $r$-Fibonacci polynomials.

\begin{theorem}
$$F_n^{[r]} = \sum_{\substack{\alpha_1, \dots \alpha_r\geq0 \\
\alpha_1+2\alpha_2\cdots +r\alpha_r = n-r+1}} \binom{\alpha_1 + \alpha_2 \dots + \alpha_r}{\alpha_1, \alpha_2 \dots,\alpha_r}x_1^{\alpha_1}x_2^{\alpha_2}\cdots x_r^{\alpha_r}$$
\end{theorem}
\begin{proof}
By \textbf{Theorem 3.3}, $F_n^{[r]}$ can be written in the form
$$F_n^{[r]} = \sum_{\substack{\alpha_1, \dots \alpha_r\geq0 \\
\sum_{i=1}^r i\alpha_i = n-r+1}} c_n(\alpha_1,...\alpha_r)x_1^{\alpha_1}x_2^{\alpha_2}\cdots x_r^{\alpha_r}$$
where $c_n(\alpha_1,...\alpha_r)\in\Z^+$ denotes the corresponding coefficient. We proceed by induction. 

When $0\leq n<r-1$, the formula is true since the sum index is the empty set. For $n=r-1$, we have that $F_{r-1}^{[r]}=1=\binom{0}{0,\cdots0}$. Thus the formula works for the first $r$ cases.

Presume that for $n\leq k$, $c_n(\alpha_1,...\alpha_r)=\binom{\alpha_1+\cdots+\alpha_2}{\alpha_1,...\alpha_r}$. Then, when $1\leq j\leq r$, by the recursive definition of $F_n^{[r]}$, 
\begin{align*}
    c_{k+1}(\alpha_1,...\alpha_r)&=\sum_{i=1}^rc_{k-i+1}(\alpha_1,...,\alpha_i-1,\alpha_{i+1},...\alpha_r)\\
    &=\sum_{i=1}^r\frac{\alpha_i(\alpha_1+\cdots+\alpha_r-1)!}{\alpha_1!\cdots\alpha_i!\cdots\alpha_r!}\\
   &=(\alpha_1+\cdots+\alpha_r-1)!\sum_{i=1}^r\frac{\alpha_i}{\alpha_1!\cdots\alpha_r!}\\
   &=\frac{(\alpha_1+\cdots+\alpha_r)!}{\alpha_1!\cdots\alpha_r!}
\end{align*}

\end{proof}

\begin{rem}
The coefficients of the monomials in $F_n^{[r]}$ correspond precisely with the number of ways to rearrange the elements of the corresponding partition.
\end{rem}

By \textbf{Theorem 3.4}, the $r$-Bonacci polynomials can be represented with a more explicit sum index by using iterated sums.

\begin{corollary}
Letting $\alpha_1=n-r+1-\sum_{k=2}^r k\alpha_k$, then

\tiny $$F_n^{[r]}=\sum_{\alpha_r=0}^{\lfloor {\frac{n-r+1}{r}} \rfloor}
\cdots
\sum_{\alpha_{r-i}=0}^{\lfloor {\frac{n-r+1-\sum_{k=0}^{i-1}(r-k)\alpha_{r-k}}{r-i}} \rfloor}
\cdots
\sum_{\alpha_{2}=0}^{\lfloor {\frac{n-r+1-\sum_{k=3}^{r}(k)\alpha_{k}}{2}} \rfloor}
 \binom{\alpha_1 + \dots + \alpha_r}{\alpha_1,\dots,\alpha_r}x_1^{\alpha_1}x_2^{\alpha_2}\cdots x_r^{\alpha_r}.$$
\end{corollary}
\begin{proof}
The iterated sums index over all sets of $\alpha_r, ... \alpha_2$ such that \newline $\sum_{k=2}^rk\alpha_k \leq n-r+1$. Thus, by letting  $\alpha_1=(n-r+1)-\sum_{k=2}^rk\alpha_k$, we can see the summations indexes over all partitions of $n-r+1$.
\end{proof}

\begin{rem}
This generalizes the combinatorial sum formula in \cite{ArsComb}.
\end{rem}

\section{Generating Functions}\label{sec:gen}
The generating function for Fibonacci numbers $f_n$ is known to be \begin{equation}
  \sum_{n=1}^\infty f_{n} x^n = \frac{x}{1-x-x^2}  
\end{equation}

for $|x+x^2| < 1$, $f_0 = 0$, and $f_1 = 1$. In this section, we develop a similar generating function for the $r$-Fibonacci polynomials and derive properties and identities connecting $r$-Fibonacci polynomials to standard combinatorial sequences. 

\subsection{Generating Function of the $r$-Fibonnaci Polynomials}
\begin{theorem}
Let $|x_1z| + |x_2z^2| + \dots + |x_rz^r| < 1$. Then the generating function of the $r$-Fibonacci polynomials converges to
$$\sum_{n=0}^\infty F_{n+r-1}^{[r]}(x_1,\dots,x_r)z^n = \frac{1}{1 - x_1z - x_2z^2 - \dots - x_rz^r}$$
\end{theorem}
\begin{proof} Because $|x_1z| + |x_2z^2| + \dots + |x_rz^r| < 1$, 
$$\frac{1}{1 - \sum_{i=1}^rx_iz^i} = \sum_{n=0}^\infty \left(\sum_{i=1}^rx_iz^i \right)^n$$ converges absolutely. By rearranging the sum, we find that
\small \begin{align*}\sum_{n=0}^\infty \left(\sum_{i=1}^rx_iz^i \right)^n &=1+(x_1z+x_2z^2+...+x_rz^r)+(x_1z+x_2z^2+...+x_rz^r)^2+...\\
&=1+(x_1)z+(x_1^2+x_2)z^2+(x_1^3+2x_1x_2+x_3)z^3+...\\
&=\sum_{n=0}^\infty \left(\sum_{\substack{\alpha_1, \dots \alpha_r\geq0 \\
\alpha_1+2\alpha_2\cdots +r\alpha_r = n}} \binom{\alpha_1 + \dots + \alpha_r}{\alpha_1,\dots,\alpha_r}x_1^{\alpha_1}x_2^{\alpha_2}\cdots x_r^{\alpha_r} \right)z^n\\
&=\sum_{n=0}^\infty F_{n+r-1}^{[r]}(x_1,\dots,x_r)z^n\end{align*}
\end{proof}

\begin{rem}
Equation (4.1) is a special case of \textbf{Theorem 4.1}, where $r=2$, and $x_1=x_2=1$.
\end{rem}

Using a similar argument, a similar generating function can be written in terms of $r$-Fibonacci polynomials where $r = \omega$.

\begin{theorem}
Let $g(z)=\sum_{k=1}^{\infty}c_kz^k$ with $|g(z)|<1$. Then

$$\sum_{n=1}^{\infty}F^{[n]}_{2n-1}(c_1,...c_n)z^n=\frac{g(z)}{1-g(z)}$$
\end{theorem}

\begin{proof}
The argument follows as in \textbf{Theorem 4.1}, but because there are always exactly $n$ variables in the coefficient of $z^n$ for the infinite-variate case, we set $r=n$.
$$1 + \sum_{n=1}^{\infty}F^{[n]}_{2n-1}(c_1,...c_n)z^n = \frac{1}{1-g(z)}$$
$$\sum_{n=1}^{\infty}F^{[n]}_{2n-1}(c_1,...c_n)z^n=\frac{g(z)}{1-g(z)}$$
\end{proof}

\subsection{Applications of the Generating Function}\hspace{1mm}

Using \textbf{Theorem 4.1} and \textbf{4.2}, we will derive identities and generating functions of other sequences. 

\begin{corollary}
If $|x_1|+...+|x_r|<1$, then $\displaystyle \lim_{n\rightarrow\infty}F_n^{[r]}=0$.
\end{corollary}

\begin{proof}
By setting $z=1$ in \textbf{Theorem 4.1}, we know that \newline $\sum_{k=0}^{\infty}F_{n+r-1}^{[r]}(x_1,...x_r)$ converges for $|x_1|+ ... +|x_r|<1$. Thus $\displaystyle \lim_{n\rightarrow\infty}F_n^{[r]}=0$.
\end{proof}

\begin{theorem}
Let $f_n=F_n^{[2]}(1,1)$ be the $n$th Fibonacci number, $p_n$ be the $n$th Pell number, and $-1<z<0$. Then, 
$$\sum_{n=1}^\infty p_nz^n = \sum_{n=1}^\infty F^{[n]}_{2n-1}(f_1,f_2,\dots, f_n)z^n$$
\end{theorem}
\begin{proof}
For $-1<z<0$, equation (4.1) gives us
$$\sum_{n=0}^\infty f_nz^n = \frac{z}{1-z-z^2} < 1$$
By \textbf{Theorem 4.3},
\begin{align*}
\frac{\sum_{n=1}^\infty f_nz^n}{1-\sum_{n=1}^\infty f_nz^n} &= \frac{\frac{z}{1-z-z^2}}{1-\frac{z}{1-z-z^2}} = \frac{z}{1-2z-z^2} = \sum_{n=1}^\infty F^{[2]}_{n}(2,1)z^n = \sum_{n=1}^\infty p_nz^n\\
&= \sum_{n=1}^\infty F^{[n]}_{2n-1}(f_1,f_2,\dots, f_n)z^n
\end{align*}
\end{proof}
\begin{rem}
$p_n \neq F^{[n]}_{2n-1}(f_1,f_2,\dots,f_n)$. The theorem is only true for $-1<z<0$, thus the coefficients of each of the generating series are not equal.
\end{rem}

We can also use the generating function to show that the $r$-Fibonacci polynomials can be manipulated to generate preference orderings, which are orderings of partitions of a set of size $n$. The Fubini numbers, denoted $a_n$, are defined to be the number of preference orderings of size $n$. 
\begin{theorem}
Let $\mathcal{O}_n{(\alpha_1,\alpha_2,\dots,\alpha_r)}$ denote the number of preference orderings for size $n$ with $\alpha_i$ partitions of size $i$. Then,
\footnotesize $$n!F^{[r]}_{n+r-1}\left(x_1,\frac{1}{2!}x_2,\frac{1}{3!}x_3,\dots,\frac{1}{r!}x_r\right)=\sum_{\substack{\alpha_1, \dots \alpha_r\geq0 \\
\sum_{i=1}^r i\alpha_i = n}} \mathcal{O}_n{(\alpha_1,\alpha_2,\dots,\alpha_r)} x_1^{\alpha_1}x_2^{\alpha_2}\dots x_r^{\alpha_r}$$
\end{theorem}
\begin{proof}
For $z<\ln(2)$, the exponential generating function for the Fubini numbers is given by \begin{equation}
    \sum_{n=1}^\infty\frac{a_n}{n!}z^n=\frac{1}{2-e^z} = \frac{1}{1-\sum_{i=1}^{\infty} \frac{1}{i!}z^i}
\end{equation}
By \textbf{Theorem 4.3},
$$
\frac{1}{1-\sum_{i=1}^{\infty} \frac{1}{i!}x_i z^i}=\sum_{n=0}^{\infty} F^{[n]}_{2n-1}\left(x_1,\frac{1}{2!}x_2,\frac{1}{3!}x_3,\dots,\frac{1}{n!}x_n\right)z^n
$$
Utilizing the composition and combinatorial ideas in the proof of (4.2) (see \cite{Fubini}), we find that adding the $x_i$ coefficients distinguishes between partitions of different sizes. Thus,
$$\frac{1}{1-\sum_{i=1}^{\infty} \frac{1}{i!}x_i z^i}=\sum_{n=0}^{\infty}\left(\sum_{\substack{\alpha_1, \dots \alpha_r\geq0 \\
\sum_{i=1}^r i\alpha_i = n}} \mathcal{O}_n{(\alpha_1,\alpha_2,\dots,\alpha_r)} x_1^{\alpha_1}x_2^{\alpha_2}\dots x_r^{\alpha_r}\right)z^n$$
By setting $x_i=0$ for $i>r$ and comparing coefficients of the generating functions we get the desired equality:
\footnotesize $$n!F^{[r]}_{n+r-1}\left(x_1,\frac{1}{2!}x_2,\frac{1}{3!}x_3,\dots,\frac{1}{r!}x_r\right)=\sum_{\substack{\alpha_1, \dots \alpha_r\geq0 \\
\sum_{i=1}^r i\alpha_i = n}} \mathcal{O}_n{(\alpha_1,\alpha_2,\dots,\alpha_r)} x_1^{\alpha_1}x_2^{\alpha_2}\dots x_r^{\alpha_r}$$
\end{proof}

\begin{corollary}
Let $a_n^r$ denote the $n$th Fubini number restricted by $r$, the number of preference orderings of a set of size $n$ with partition size of max size $r$. Then,
$$a_n^r=F^{[r]}_{n+r-1}\left(1,\frac{1}{2!},\frac{1}{3!},\dots,\frac{1}{r!}\right)n!$$
\end{corollary}
\begin{proof}
Setting $x_1=x_2=...x_r=1$ in \textbf{Theorem 4.5} gives the desired equality by the definition of $a_n^r$.
\end{proof}
\begin{rem}
Since $a_n=a_n^n$,
$$a_n=F^{[n]}_{2n-1}\left(1,\frac{1}{2!},\frac{1}{3!},\dots,\frac{1}{n!}\right)n!.$$
In the following section, we will derive the following identity in terms of the complete ordinary Bell polynomials:
$$a_n=\hat{B}_n\left(1,\frac{1}{2!},\frac{1}{3!},\dots,\frac{1}{n!}\right)n!$$
\end{rem}

\section{Bell Polynomials}\label{sec:bell}

In this section, the explicit sum formula and the generating function are used to draw an identity to Bell polynomials, which will be used to give a proof of a combinatorial identity. 

\subsection{Ordinary Bell Polynomials} \

Firstly, recall the definition of the ordinary Bell polynomials. 
\begin{definition}
The \textbf{Partial Ordinary Bell Polynomial} is defined as 
$$\hat{B}_{n,k} = \sum_{\substack{j_1+j_2+\dots +j_{n-k+1} = k \\ j_1+2j_2+\dots + (n-k+1)j_{n-k+1} = n}} \frac{k!}{j_1!j_2!\dots j_{n-k+1}!}x_1^{j_1}x_2^{j_2}\dots x_{n-k+1}^{j_{n-k+1}}$$
\end{definition}
\begin{definition}
The \textbf{Complete Ordinary Bell Polynomial} is defined as
$$\hat{B}_n = \sum_{k=1}^n \hat{B}_{n,k}$$
\end{definition}

Note that the coefficients in the definition of $\hat{B}_{n,k}$ is almost identical to that of $F_n^{[r]}$ in \textbf{Theorem 3.4}. The next theorem proves how they are related. 
\begin{theorem}
$$\hat{B}_n = F_{2n-1}^{[n]}(x_1,\dots,x_n)$$
\end{theorem}
\begin{proof}
$$\hat{B}_n = \sum_{k=1}^n \hat{B}_{n,k} = \sum_{k=1}^n\sum_{\substack{j_1+j_2+\dots +j_{n-k+1} = k \\ j_1+2j_2+\dots + (n-k+1)j_{n-k+1} = n}} \frac{k!}{j_1!j_2!\dots j_{n-k+1}!}x_1^{j_1}x_2^{j_2}\dots x_{n-k+1}^{j_{n-k+1}}$$
$$=\sum_{j_1+2j_2+\dots + nj_{n} = n} \frac{(j_1+j_2+\dots +j_n)!}{j_1!j_2!\dots j_n!}x_1^{j_1}x_2^{j_2}\dots x_n^{j_n}$$
$$=F^{[n]}_{2n-1}(x_1,\dots,x_n)$$
\end{proof}
\begin{rem} By letting $x_i=0$ for all $i>r$, 
$$\hat{B}_n(x_1,\dots,x_r,0,\dots) = F_{n+r-1}^{[r]}(x_1,\dots,x_r)$$
\end{rem}
We see then that $r$-Fibonacci polynomials are generalizations of the complete ordinary Bell polynomials. 

\subsection{Exponential Bell Polynomials} \

Recall the definition of the exponential Bell polynomial.

\begin{definition}
The \textbf{Partial Exponential Bell Polynomial} is defined as
$$B_{n,k} = \sum_{\substack{j_1+j_2+\dots +j_{n-k+1} = k \\ j_1+2j_2+\dots + (n-k+1)j_{n-k+1} = n}} \frac{n!}{j_1!j_2!\dots j_{n-k+1}!}x_1^{j_1}x_2^{j_2}\dots x_{n-k+1}^{j_{n-k+1}}$$
\end{definition}

\begin{theorem}
$$\sum_{k=1}^n k!B_{n,k}(x_1, 2x_2, 3!x_3, \dots, r!x_r, 0, \dots) = n!F_{n+r-1}^{[r]}$$
\end{theorem}
\begin{proof}
By taking the $n$th derivative of the generating function for $r$-Fibonacci polynomials, we find that 
$$\sum_{i=0}^{\infty}\frac{(n+i)!}{i!}F_{i+n+r-1}z^i = \frac{d^n}{dz^n} \left(\frac{1}{g(z)} \right)$$
where $g(z) = 1-\sum_{i=0}^r x_iz^i$. Recall that Faà di Bruno's formula (see \cite{faa_formula}) tells us
$$\frac{d^n}{dz^n}f(g(z)) = \sum_{k=1}^n f^{(k)}(g(z)) B_{n,k}(g'(z), g''(z), \dots, g^{(n-k+1)}(z))$$
If we set $f(z) = \frac{1}{z}$, then we find
$$\sum_{i=0}^{\infty}\frac{(n+i)!}{i!}F_{i+n+r-1}z^i = \sum_{k=1}^n\frac{(-1)^k k!}{g(x)^k} B_{n,k}(g'(z), \dots, g^{(n-k+1)}(z))$$
If $z=0$, 
$$\sum_{k=1}^n (-1)^k k!B_{n,k}(-x_1, -2x_2, -3!x_3, \dots, -r!x_r, 0, \dots)$$
$$=\sum_{k=1}^n k!B_{n,k}(x_1, 2x_2, 3!x_3, \dots, r!x_r, 0, \dots) = n!F_{n+r-1}^{[r]}$$
\end{proof}

Utilizing the above identity, we demonstrate a new proof of a known relation between the Fubini numbers and the Stirling numbers of the second kind. 

\begin{corollary}
Let $a_n$ denote the Fubini numbers. Then, 
$$\sum_{k=1}^n k!\stirling{n}{k} = a_n$$
\end{corollary}
\begin{proof}
\textbf{Theorem 4.4} gives the following equality
$$n!F_{2n-1}^{[n]}\left(\frac{1}{1!}, \frac{1}{2!}, \frac{1}{3!}, \dots, \frac{1}{n!}\right) = a_n$$

Additionally, we have the following identity between the partial exponential Bell polynomials and Stirling numbers (see \cite{AdvancedComb}):

$$B_{n,k}(1,1,...1)=\stirling{n}{k} $$

By \textbf{Theorem 5.5}, we find 
$$a_n = n!F_{2n-1}^{[n]}\left(\frac{1}{1!}, \frac{1}{2!}, \frac{1}{3!}, \dots, \frac{1}{n!}\right) = \sum_{k=1}^n k!B_{n,k}(1, 1, 1, \dots) = \sum_{k=1}^n k!\stirling{n}{k}$$
\end{proof}
This corollary shows how the $r$-Fibonacci polynomials can be used as a tool to prove combinatorial identities in new manner.

\medskip

\end{document}